\documentclass[11pt,a4paper]{article}

\usepackage{graphicx,latexsym,euscript,makeidx,color,bm}
\usepackage{amsmath,amsfonts,amssymb,amsthm,thmtools,mathrsfs,enumerate}
\usepackage[colorlinks,linkcolor=blue,anchorcolor=green,citecolor=red]{hyperref}

\usepackage{geometry}
\def\dbE{\mathbb{E}}     
\def\dbF{\mathbb{F}}   \def\cF{{\cal F}}  
     
\def\dbH{\mathbb{H}}

\def\dbP{\mathbb{P}}     
   \def\cQ{{\cal Q}}  
\def\dbR{\mathbb{R}} \def\sR{\mathscr{R}}    
\def\dbS{\mathbb{S}}     
     
   \def\cU{{\cal U}}

\def\ss{\smallskip}      \def\lt{\left}       \def\hb{\hbox}
\def\ms{\medskip}        \def\rt{\right}      \def\ae{\hbox{\rm a.e.}}
        \def\lan{\langle}    \def\as{\hbox{\rm a.s.}}
\def\ds{\displaystyle}   \def\ran{\rangle}    \def\tr{\hbox{\rm tr$\,$}}
\def\ts{\textstyle}      \def\llan{\lt\lan}   
\def\no{\noindent}       \def\rran{\rt\ran}   
\def\ns{\noalign{\ss}}

\def\rf{\eqref}            \def\hp{\hphantom}
\def\deq{\triangleq}     \def\({\Big (}       \def\nn{\nonumber}
\def\les{\leqslant}      \def\){\Big )}       
\def\ges{\geqslant}      \def\[{\Big[}        
          \def\]{\Big]}        
      \def\q{\quad}        
         \def\qq{\qquad}      \def\1n{\negthinspace}
\def\cd{\cdot}           \def\2n{\1n\1n}      \def\3n{\1n\1n\1n}

              \def\Om{\Omega}  
         \def\D{\Delta}   \def\d{\delta}   \def\F{\Phi}     
\def\z{\zeta}         \def\Th{\Theta}      \def\si{\sigma}
\def\e{\varepsilon}     \def\l{\lambda}        
    \def\t{\tau}     \def\f{\varphi}  \def\i{\infty}   

\def\ba{\begin{array}}                \def\ea{\end{array}}
\def\bel{\begin{equation}\label}      \def\ee{\end{equation}}

\newtheorem{theorem}{Theorem}[section]
\newtheorem{definition}[theorem]{Definition}
\newtheorem{proposition}[theorem]{Proposition}

\newtheorem{lemma}[theorem]{Lemma}
\newtheorem{remark}[theorem]{Remark}
\newtheorem{example}[theorem]{Example}
\newtheorem{taggedassumptionx}{}
\newenvironment{taggedassumption}[1]
 {\renewcommand\thetaggedassumptionx{#1}\taggedassumptionx}
 {\endtaggedassumptionx}
\makeatletter
   
   \@addtoreset{equation}{section}
\makeatother
  \allowdisplaybreaks[4]

\begin{document}

\title{\bf Weak Closed-Loop Solvability of Stochastic Linear-Quadratic Optimal Control Problems}
\author{Jingrui Sun\thanks{Department of Mathematics, University of Central Florida,
                           Orlando, FL 32816, USA (Email: {\tt sjr@} {\tt mail.ustc.edu.cn}).}\, ,\,
        Hanxiao Wang\thanks{School of Mathematical Sciences, Fudan University,
                    Shanghai 200433, China (Email: {\tt hxwang14@} {\tt fudan.edu.cn}). This author is supported in part by the China Scholarship Council, while visiting University of Central Florida.}\, ,\,   and  \
       Jiongmin Yong\thanks{Department of Mathematics, University of Central Florida,
                           Orlando, FL 32816, USA (Email: {\tt jiongmin.yong@ucf.edu}).
                           This author is supported in part by NSF DMS-1812921.}}

\maketitle

\no\bf Abstract. \rm
Recently it has been found that for a stochastic linear-quadratic
optimal control problem (LQ problem, for short) in a finite horizon, open-loop solvability is strictly weaker than closed-loop solvability which is equivalent to the regular solvability of the corresponding Riccati equation. Therefore, when an LQ problem is merely open-loop solvable not closed-loop solvable, which is possible, the usual Riccati equation approach will fail to produce a state feedback representation of open-loop optimal controls. The objective of this paper is to introduce and investigate the notion of weak closed-loop optimal strategy for LQ problems so that its existence is equivalent to the open-loop solvability of the LQ problem. Moreover, there is at least one open-loop optimal control admitting a state feedback representation. Finally, we present an example to illustrate the procedure for finding weak closed-loop optimal strategies.

%

\ms

\no\bf Keywords. \rm
stochastic linear-quadratic optimal control, Riccati equation, open-loop solvability,
weak closed-loop solvability, state feedback.

\ms

\no\bf AMS subject classifications. \rm 49N10, 49N35, 93E20.

\section{Introduction}\label{Sec:Introduction}

Let $(\Om,\cF,\dbP)$ be a complete probability space on which a standard one-dimensional Brownian motion $W(\cd)=\{W(t);0\les t<\i\}$ is defined,
and let $\dbF=\{\cF_t\}_{t\ges0}$ be the natural filtration of $W(\cd)$ augmented by all the $\dbP$-null sets in $\cF$. Let $0\les t<T$ and consider the following controlled linear stochastic differential equation (SDE, for short) on the finite horizon $[t,T]$:
\bel{state}\left\{\begin{aligned}
   dX(s) &=[A(s)X(s)+ B(s)u(s)+ b(s)]ds\\
         &\hp{=\ } +[C(s)X(s) + D(s)u(s)+\si(s)]dW(s),\q s\in[t,T],\\
     X(t)&= x,
\end{aligned}\right.\ee
where $A,C:[0,T]\to\dbR^{n\times n}$, $B,D:[0,T]\to\dbR^{n\times m}$ are given deterministic functions, called the {\it coefficients} of the {\it state equation} \rf{state}; $b,\si:[0,T]\times\Om\to\dbR^n$ are $\dbF$-progressively measurable processes, called the {\it nonhomogeneous terms}; and $(t,x)\in[0,T)\times\dbR^n$ is called the {\it initial pair}.
Here, $\dbR^n$ is the usual $n$-dimensional Euclidean space consisting of all $n$-tuple of real numbers, and $\dbR^{n\times m}$ is the set of all $n\times m$ real matrices. In the above, the process $u(\cd)$, which belongs to the following space:
\begin{align*}
\cU[t,T]&\equiv L_\dbF^2(t,T;\dbR^m)= \Big\{u:[t,T]\times\Om\to\dbR^m~|~u\hb{~is $\dbF$-progressively measurable}, \\
                     &\hp{\big\{}\ts \qq\qq\qq\qq\qq\qq\qq\qq\q\ds\hb{~and~}\dbE\int^T_t|u(s)|^2ds<\i\Big\},
\end{align*}
is called the {\it control process}, and the solution $X(\cd)$ of \rf{state} is called the {\it state process} corresponding to $(t,x)$ and $u(\cd)$.
According to the standard results of SDEs, under appropriate conditions, for any initial pair $(t,x)$ and any control $u(\cd)\in\cU[t,T]$, equation \rf{state} admits a unique (strong) solution
$X(\cd)\equiv X(\cd\,;t,x,u(\cd))$ which is continuous and square-integrable. To measure the performance of the control $u(\cd)$, we introduce the following quadratic {\it cost functional}:
\begin{eqnarray}\label{cost}
J(t,x;u(\cd)) &\3n=\3n& \dbE\big\{\lan GX(T),X(T)\ran + 2\lan g,X(T)\ran \nn\\
         &\3n~\3n& +\int_t^T\Bigg[\llan\begin{pmatrix}Q (s)& S(s)^\top \\ S(s) & R(s)\end{pmatrix}
                                 \begin{pmatrix}X(s) \\ u(s)\end{pmatrix},
                                 \begin{pmatrix}X(s) \\ u(s)\end{pmatrix}\rran \nn\\
\3n&~\3n& +\,2\llan\begin{pmatrix}q(s) \\ \rho(s)\end{pmatrix},
                   \begin{pmatrix}X(s) \\ u(s)\end{pmatrix}\rran\Bigg]ds\Bigg\},
\end{eqnarray}
where $G\in\dbR^{n\times n}$ is a symmetric constant matrix;
$g$ is an $\cF_T$-measurable random variable taking values in $\dbR^n$;
$Q:[0,T]\to\dbR^{n\times n}$, $S:[0,T]\to\dbR^{m\times n}$, and $R:[0,T]\to\dbR^{m\times m}$
are deterministic functions with $Q$ and $R$ being symmetric;
and $q:[0,T]\times\Om\to\dbR^n$, $\rho:[0,T]\times\Om\to\dbR^m$ are $\dbF$-progressively measurable processes.
In the above, $M^\top$ stands for the transpose of a matrix $M$. The problem that we are going to study is the following:

\ms

{\bf Problem (SLQ).} For any given initial pair $(t,x)\in[0,T)\times\dbR^n$,
find a control $\bar u(\cd)\in \cU[t,T]$ such that
\bel{inf J}J(t,x;\bar u(\cd))\les J(t,x;u(\cd)),\q\forall u(\cd)\in \cU[t,T].\ee

\ss

The above is called a {\it stochastic linear-quadratic (LQ, for short) optimal control problem}. Any $\bar u(\cd)\in L_\dbF^2(t,T;\dbR^m)$ satisfying \eqref{inf J} is called an {\it open-loop optimal control}
of Problem (SLQ) for the initial pair $(t,x)$; the corresponding state process $\bar X(\cd)\equiv X(\cd\,;t,x,\bar u(\cd))$ is called an {\it optimal state process}; and the function $V(\cd\,,\cd)$ defined by
$$V(t,x) \deq \inf_{u(\cd)\in\cU[t,T]}J(t,x;u(\cd));\qq(t,x)\in[0,T]\times\dbR^n$$
is called the {\it value function} of Problem (SLQ).

\ms

Note that in the special case when $b(\cd),\si(\cd),g,q(\cd),\rho(\cd)=0$,
the state equation \rf{state} and the cost functional \rf{cost} become
\bel{state0}\left\{\begin{aligned}
   dX(s) &=[A(s)X(s)+ B(s)u(s)]ds + [C(s)X(s) + D(s)u(s)]dW(s),\q s\in[t,T],\\
     X(t)&= x,
\end{aligned}\right.\ee
and
\bel{cost0}
J^0(t,x;u(\cd)) = \dbE\lt\{\lan GX(T),X(T)\ran
       +\int_t^T\llan\1n\begin{pmatrix}Q (s)& S(s)^\top \\ S(s) & R(s)\end{pmatrix}\1n
                        \begin{pmatrix}X(s) \\ u(s)\end{pmatrix}\1n,
                        \begin{pmatrix}X(s) \\ u(s)\end{pmatrix}\1n\rran ds\rt\}, \ee
respectively. We refer to the problem of minimizing \rf{cost0} subject to \rf{state0} as the {\it homogeneous LQ problem associated with Problem (SLQ)}, denoted by Problem (SLQ)$^0$. The value function of Problem (SLQ)$^0$ will be denoted by $V^0(\cd\,,\cd)$.

\ms

LQ optimal control is a classical and fundamental problem in control theory,
whose history can be traced back to the works of Bellman--Glicksberg--Gross \cite{Belman-Gicksberg-Gross 1958}, Kalman \cite{Kalman 1960}, and Letov \cite{Letov 1961}. These works were concerned with deterministic cases, i.e., the state equation is a linear ordinary differential equation (ODE, for short), and all the involved functions are deterministic. Stochastic LQ problems were firstly studied by Wonham \cite{Wonham 1968} in 1968.
Later, Bismut \cite{Bismut 1976} carried out a detailed analysis for stochastic LQ optimal control with random coefficients. See also some follow-up works of Davis \cite{Davis 1977}, Bensoussan \cite{Bensoussan 1982} and Tang \cite{Tang 2003,Tang 2015}.
In the classical setting, it is typically assumed that the cost functional has positive semi-definite weighting matrices for the control and the state.
Namely, the following assumption was taken for granted: For some constant $\d>0$,
\bel{standard-condition} G\ges0, \q R(s)\ges\d I, \q Q(s)-S(s)^\top R(s)^{-1}S(s)\ges0,\q~s\in[0,T]. \ee
Such a condition ensures that the LQ problem admits a unique open-loop optimal control and that the following associated Riccati equation has a unique positive definite solution on $[0,T]$ (with the argument $s$ being suppressed):
\bel{Ric}\left\{\begin{aligned}
  & \dot P+PA+A^\top P+C^\top PC+Q \\
  & \hp{\dot P} -(PB+C^\top PD+S^\top)(R+D^\top PD)^{-1}(B^\top P+D^\top PC+ S)=0, \\
  & P(T)=G.
\end{aligned}\right.\ee
Further, the unique open-loop optimal control can be expressed as a linear feedback of the current state via the solution to \rf{Ric} (see \cite{Bismut 1976} or \cite[Chapter 6]{Yong-Zhou 1999}). It is noteworthy that \rf{standard-condition} is a quite strong set of conditions for the existence of an open-loop optimal control. Later developments show that a stochastic LQ problem might still admit an open-loop optimal control even if the control weight $R(\cd)$ is negative definite; see \cite{Chen-Li-Zhou 1998,Lim-Zhou 1999,Chen-Yong 2000,Chen-Zhou 2000,Rami-Moore-Zhou 2001,Chen-Yong 2001}
for some relevant works on the so-called {\it indefinite} stochastic LQ control problem.

\ms

Recently, Sun--Yong \cite{Sun-Yong 2014} and Sun--Li--Yong \cite{Sun-Li-Yong 2016} investigated the open-loop and closed-loop solvabilities of stochastic LQ problems. It was shown that the existence of an open-loop optimal control (open-loop solvability of LQ problem) is equivalent to the solvability of the associated optimality system (which is a constrained forward-backward SDE, abbreviated as FBSDE), and that the existence of a closed-loop optimal strategy (closed-loop solvability of LQ problem) is equivalent to the {\it regular} solvability of the following {\it generalized} Riccati equation (GRE, for short):
\bel{GRE}\left\{\begin{aligned}
  & \dot P+PA+A^\top P+C^\top PC+Q \\
  & \hp{\dot P} -(PB+C^\top PD+S^\top)(R+D^\top PD)^\dag(B^\top P+D^\top PC+ S)=0, \\
%
  & P(T)=G,
\end{aligned}\right.\ee
where $M^\dag$ denotes the Moore-Penrose pseudoinverse of a matrix $M$. In the above, the argument $s$ is again suppressed; We will do that in the following, as long as no ambiguity will arise. It was found (\cite{Sun-Yong 2014,Sun-Li-Yong 2016}) that the existence of a closed-loop optimal strategy implies the existence of an open-loop optimal control, but not vice versa. Thus, there are some LQ problems that are open-loop solvable, but not closed-loop solvable; for such problems, one could not expect to get a regular solution (which does not exist) to the associated GRE \rf{GRE}, so that the state feedback representation of the open-loop optimal control might be impossible. To be more convincing, let us look at the following simple example.

\begin{example}\label{ex-1.1}\rm
Consider the one-dimensional state equation
$$\left\{\begin{aligned}
   dX(s) &= [-2X(s)+u(s)]ds + 2X(s)dW(s), \q s\in[t,1],\\
    X(t) &= x,
\end{aligned}\right.$$
and the nonnegative cost functional
$$ J(t,x;u(\cd))=\dbE |X(1)|^2. $$
In this example, the associated GRE reads
\bel{example-Ric} \dot P(s)=0,\q s\in[0,1]; \qq P(1)=1. \ee
Clearly, $P(s)\equiv1$ is the unique solution of \rf{example-Ric}. From \cite{Sun-Li-Yong 2016}, we know that such a solution is not regular. A usual Riccati equation approach specifies the corresponding state feedback control as follows (noting that $R(\cd)=0$, $D(\cd)=0$, and $0^\dag=0$):
$$ u^*(s) \deq -\big[R(s)+D(s)^\top P(s)D(s)\big]^\dag\big[B(s)^\top P(s)+D(s)^\top P(s)C(s)+ S(s)\big]X(s) \equiv0,$$
which is {\it not}  open-loop optimal for any nonzero initial state $x$.
In fact, let $(t,x)\in[0,1)\times\dbR$ be an arbitrary but fixed initial pair with $x\ne0$. By the variation of constants formula, the state process $X^*(\cd)$ corresponding to $(t,x)$ and $u^*(\cd)$ is given by
$$X^*(s)=e^{2W(s)-4s}x,\qq s\in[t,1].$$
Hence,
$$J(t,x;u^*(\cd))=\dbE|X^*(1)|^2=x^2>0.$$
On the other hand, let $\bar u(\cd)$ be the control defined by
$$\bar u(s)\equiv {x\over t-1}e^{2W(s)-4s}, \q s\in[t,1]. $$
By the variation of constants formula, the state process $\bar X(\cd)$ corresponding to $(t,x)$ and $\bar u(\cd)$ is given by
$$\ba{ll}
\ns\ds\bar X(s)=e^{2W(s)-4s}\[x+\int_t^se^{-2W(r)+4r}\,\bar u(r)dr\]\\
\ns\ds\qq~=e^{2W(s)-4s}\[x+{s-t\over t-1}x\],\qq s\in[t,1],\ea $$
which satisfies $\bar X(1)=0$. Hence,
$$ J(t,x;\bar u(\cd)) =\dbE|\bar X(1)|^2 =0 <J(t,x;u^*(\cd)). $$
Since the cost functional is nonnegative, we see that $\bar u(\cd)$ is open-loop optimal for the initial pair $(t,x)$, but $u^*(\cd)$ is not.
\end{example}

The above example suggests that the usual solvability of the generalized Riccati equation \rf{GRE} may not be helpful in handling open-loop solvability of certain stochastic LQ problems. It is then natural to ask: {\sl When Problem (SLQ) is merely open-loop solvable, not closed-loop solvable; is it still possible to get a linear state feedback representation for an open-loop optimal control?} The objective of this paper is to tackle this problem. We shall provide an alternative characterization of the open-loop solvability of Problem (SLQ) using the perturbation approach introduced in \cite{Sun-Li-Yong 2016}. We point out here that our result, which avoids the subsequence extraction, is a sharpened version of
\cite[Theorem 6.2]{Sun-Li-Yong 2016}. In order to obtain a linear state feedback representation of open-loop optimal control for Problem (SLQ),
we introduce the notion of weak closed-loop strategies. This notion is a slight extension of the closed-loop strategy developed in \cite{Sun-Yong 2014,Sun-Li-Yong 2016}. We shall prove that as long as Problem (SLQ) is open-loop solvable, there always exists a weak closed-loop strategy
whose outcome is an open-loop optimal control. Note that it might be that the open-loop optimal control is not unique and we are able to represent one of them in the state feedback form.

\ms

The rest of the paper is organized as follows. In Section \ref{Sec:Preliminaries}, we collect some preliminary results and introduce a few elementary notions for Problem (SLQ). Section \ref{Sec:3} is devoted to the study of open-loop solvability by a perturbation method.
In section \ref{Sec:4}, we show how to obtain a weak closed-loop optimal strategy and establish the equivalence between open-loop and weak closed-loop solvability. An example is presented in Section \ref{Sec:Example} to illustrate the results we obtained.

\section{Preliminaries}\label{Sec:Preliminaries}

Throughout this paper, and recall from the previous section, $M^\top$ stands for the transpose of a matrix $M$, $\tr(M)$ the trace of $M$, $\dbR^{n\times m}$ the Euclidean space consisting of $(n\times m)$ real matrices, endowed with the Frobenius inner product $\lan M,N\ran\mapsto\tr[M^\top N]$. We shall denote by $I_n$ the identity matrix of size $n$ and by $|M|$ the Frobenius norm of a matrix $M$. Let $\dbS^n$ be the subspace of $\dbR^{n\times n}$ consisting of symmetric matrices. For $M,N\in\dbS^n$, we use the notation $M\ges N$ (respectively, $M>N$) to indicate that $M-N$ is positive
semi-definite (respectively, positive definite). Let $[t,T]$ be a subinterval of $[0,\i)$ and $\dbH$ be a Euclidean space (which could be $\dbR^n$, $\dbR^{n\times m}$, $\dbS^n$, etc.). We further introduce the following spaces of functions and processes:
\begin{align*}
C([t,T];\dbH):
   &\hb{~~the space of $\dbH$-valued, continuous functions on $[t,T]$}. \\
L^p(t,T;\dbH):
   &\hb{~~the space of $\dbH$-valued functions that are $p$th $(1\les p\les\i)$} \\
   &\hb{~~power Lebesgue integrable on $[t,T]$}.\\
L^2_{\cF_T}(\Om;\dbH):
   &\hb{~~the space of $\cF_T$-measurable, $\dbH$-valued random variables $\xi$} \\
   &\hb{~~such that $\dbE|\xi|^2<\i$}. \\
L^2_\dbF(\Om;L^1(t,T;\dbH)):
   &\hb{~~the space of $\dbF$-progressively measurable, $\dbH$-valued processes} \\
   &\hb{~~$\f:[t,T]\times\Om\to\dbH$ such that $\dbE\big[\ds\int_t^T|\f(s)|ds\big]^2<\i$}.\\
L_\dbF^2(t,T;\dbH):
   &\hb{~~the space of $\dbF$-progressively measurable, $\dbH$-valued processes} \\
   &\hb{~~$\f:[t,T]\times\Om\to\dbH$ such that $\ds\dbE\int_t^T|\f(s)|^2ds<\i$}. \\
L_\dbF^2(\Om;C([t,T];\dbH)):
   &\hb{~~the space of $\dbF$-adapted, continuous, $\dbH$-valued processes} \\
   &\hb{~~$\f:[t,T]\times\Om\to\dbH$ such that $\dbE\big[\sup_{s\in[t,T]}|\f(s)|^2\big]<\i$}.
\end{align*}

\ss

To guarantee the well-posedness of the state equation \rf{state}, we adopt the following assumption:
\begin{taggedassumption}{(A1)}\label{ass:A1}
The coefficients and the nonhomogeneous terms of \rf{state} satisfy
$$\left\{\begin{aligned}
   A(\cd) &\in L^1(0,T;\dbR^{n\times n}),     &&&   B(\cd) &\in L^2(0,T;\dbR^{n\times m}),\\
   C(\cd) &\in L^2(0,T;\dbR^{n\times n}),     &&&   D(\cd) &\in L^\i(0,T;\dbR^{n\times m}),\\
   b(\cd) &\in L^2_\dbF(\Om;L^1(0,T;\dbR^n)), &&& \si(\cd) &\in L_\dbF^2(0,T;\dbR^n).
\end{aligned}\right. $$
\end{taggedassumption}

The following result, whose proof can be found in \cite[Proposition 2.1]{Sun-Yong 2014},
establishes the well-posedness of the state equation under the assumption \ref{ass:A1}.

\begin{lemma}\label{lmm:well-posedness-SDE}
Let {\rm\ref{ass:A1}} hold. Then for any initial pair $(t,x)\in[0,T)\times\dbR^n$ and control $u(\cd)\in\cU[t,T]$, the state equation \rf{state} admits a unique solution $X(\cd)\equiv X(\cd\,;t,x,u(\cd))$. Moreover, there exists a constant $K>0$, independent of $(t,x)$ and $u(\cd)$, such that
$$ \dbE\lt[\sup_{t\les s\les T}|X(s)|^2\rt]
\les K\dbE\lt[|x|^2+\lt(\int_t^T|b(s)|ds\rt)^2 + \int_t^T|\si(s)|^2ds + \int^T_t|u(s)|^2ds\rt].$$
\end{lemma}

To ensure that the random variables in the cost functional \rf{cost} are integrable,
we assume the following holds:
\begin{taggedassumption}{(A2)}\label{ass:A2}
The weighting coefficients in the cost functional satisfy
$$\left\{\2n\ba{ll}
\ns\ds Q(\cd)\in L^1(0,T;\dbS^n),\qq S(\cd)\in L^2(0,T;\dbR^{m\times n}),
\qq R(\cd)\in L^\infty(0,T;\dbS^m), \\
\ns\ds q(\cd)\in L^2_\dbF(\Om;L^1(0,T;\dbR^n)),\q\rho(\cd)\in L_\dbF^2(0,T;\dbR^m),\q g\in L^2_{\cF_T}(\Om;\dbR^n),\q G\in\dbS^n.\ea\right.$$

\end{taggedassumption}
\begin{remark}\rm
Suppose that \ref{ass:A1} holds. Then according to \autoref{lmm:well-posedness-SDE}, for any initial pair $(t,x)\in[0,T)\times\dbR^n$ and any control $u(\cd)\in\cU[t,T]$,
equation \rf{state} admits a unique (strong) solution $X(\cd)\equiv X(\cd\,;t,x,u(\cd))$
which belongs to the space $L_\dbF^2(\Om;C([t,T];\dbR^n))$.
If, in addition, \ref{ass:A2} holds, then the random variables on the right-hand side of \rf{cost} are integrable and hence Problem (SLQ) is well-posed. It is worth pointing out that we do not impose any positive-definiteness/nonnegativeness conditions on $Q(\cd)$, $R(\cd)$, and $G$.
\end{remark}

Now we recall some basic notions of stochastic LQ optimal control problems.

\begin{definition}\rm
Problem (SLQ) is said to be
\begin{enumerate}[(i)]
\item ({\it uniquely}) {\it open-loop solvable at $(t,x)\in[0,T)\times\dbR^n$} if
      there exists a (unique) $\bar u(\cd)\equiv \bar u(\cd\,;t,x)\in \cU[t,T]$ (depending on $(t,x)$) such that
      \bel{open-optimal} J(t,x;\bar u(\cd))\les J(t,x;u(\cd)),\q\forall u(\cd)\in\cU[t,T]. \ee
      Such a $\bar u(\cd)$ is called an {\it open-loop optimal control for $(t,x)$}.
      %
      %
\item ({\it uniquely}) {\it open-loop solvable} if it is (uniquely) open-loop solvable at any initial pair $(t,x)\in[0,T)\times\dbR^n$.
\end{enumerate}
\end{definition}

\begin{definition}\rm
Let $\Th:[t,T]\to\dbR^{m\times n}$ be a deterministic function and $v:[t,T]\times\Om\to\dbR^m$
be an $\dbF$-progressively measurable process.
\begin{enumerate}[(i)]
\item We call $(\Th(\cd),v(\cd))$ a {\it closed-loop strategy} on $[t,T]$ if $\Th(\cd)\in L^2(t,T;\dbR^{m\times n})$
      and $v(\cd)\in L_\dbF^2(t,T;\dbR^m)$; that is,
      $$\int_t^T|\Th(s)|^2ds<\i, \q \dbE\int_t^T|v(s)|^2ds<\i.$$
      The set of all closed-loop strategies $(\Th(\cd),v(\cd))$ on $[t,T]$ is denoted by $\cQ[t,T]$.

\item A closed-loop strategy $(\Th^*(\cd),v^*(\cd))\in\cQ[t,T]$
      is said to be {\it optimal} on $[t,T]$ if
      \bel{closed-optimal*}\ba{ll}
       \ns\ds J(t,x;\Th^*(\cd)X^*(\cd)+v^*(\cd))\les J(t,x;\Th(\cd)X(\cd)+v(\cd)),\\
       \ns\ds\qq\qq\qq\qq\qq\qq\forall x\in\dbR^n,~\forall(\Th(\cd),v(\cd))\in\cQ[t,T],\ea \ee
      where $X^*(\cd)$ is the solution to the {\it closed-loop system} under $(\Th^*(\cd),v^*(\cd))$ (with the argument s
      suppressed in the coefficients and non-homogeneous terms):
      \bel{closed-syst*}\left\{\begin{aligned}
         dX^*(s) &=\big[\big(A+B\Th^*\big)X^*(s)+Bv^*+b\big]ds \\
                 &\hp{=\ } +\big[\big(C+D\Th^*\big)X^*(s)+Dv^*+\si\big]dW(s),\qq s\in[t,T],\\
          X^*(t) &=x,
      \end{aligned}\right.\ee
      and $X(\cd)$ is the solution to the following closed-loop system under $(\Th(\cd),v(\cd))$:
      \bel{closed-syst}\left\{\begin{aligned}
         dX(s) &=\big[\big(A+B\Th\big)X(s)+Bv+b\big]ds \\
                 &\hp{=\ } +\big[\big(C+D\Th\big)X(s)+Dv+\si\big]dW(s),\qq\q s\in[t,T],\\
          X(t) &=x.
      \end{aligned}\right.\ee

\item If for any $t\in[0,T)$, a closed-loop optimal strategy (uniquely) exists on $[t,T]$,
      we say Problem (SLQ) is ({\it uniquely}) {\it closed-loop solvable}.
\end{enumerate}
\end{definition}

From \cite[Proposition 3.3]{Sun-Yong 2014}, we know that $(\Th^*(\cd),v^*(\cd))$ is a closed-loop optimal strategy on $[t,T]$ if and only if
\bel{closed-optimal}J(t,x;\Th^*(\cd)X^*(\cd)+v^*(\cd))\les J(t,x;u(\cd)),\q\forall x\in\dbR^n,\q\forall u(\cd)\in\cU[t,T].\ee
Comparing \rf{closed-optimal} with \rf{open-optimal}, we see that if $(\Th^*(\cd),v^*(\cd))$ is a closed-loop optimal strategy of Problem (SLQ) on $[t,T]$, then the {\it outcome} $u^*(\cd)\equiv\Th^*(\cd)X^*(\cd)+v^*(\cd)$ is an open-loop optimal control of Problem (SLQ) for the corresponding initial pair $(t,X^*(t))$. This means that the closed-loop solvability implies the open-loop solvability. However, there are examples showing that the reverse implication is not necessarily true, namely, it is possible that an LQ problem is only open-loop solvable, not closed-loop solvable; see \cite[Example 7.1]{Sun-Li-Yong 2016}. In another word, it is possible that some open-loop optimal control $\bar u(\cd)$ is not an outcome of some (regular) closed-loop optimal strategy. Since the closed-loop representation is very important in practice, we naturally ask: Is it possible for some (not necessarily all) open-loop optimal controls, one can find less regular closed-loop representation? Motivated by this, we introduce the following notion.

\begin{definition}\rm\label{def-wcloop}
Let $\Th:[t,T)\to\dbR^{m\times n}$ be a locally square-integrable deterministic function and $v:[t,T)\times\Om\to\dbR^m$ be a locally square-integrable $\dbF$-progressively measurable process; that is, $\Th(\cd)$ and $v(\cd)$ are such that for any $T^\prime\in[t,T)$
$$\int_t^{T'}|\Th(s)|^2ds<\i, \q \dbE\int_t^{T'}|v(s)|^2ds<\i.$$
\begin{enumerate}[(i)]
\item We call $(\Th(\cd),v(\cd))$ a {\it weak closed-loop strategy} on $[t,T)$ if for any initial state $x\in\dbR^n$,
      the outcome $u(\cd)\equiv\Th(\cd)X(\cd)+v(\cd)$ of $(\Th(\cd),v(\cd))$ belongs to $\cU[t,T]\equiv L_\dbF^2(t,T;\dbR^m)$, where $X(\cd)$ is the solution to the {\it weak closed-loop system}
      \bel{weak-syst}\left\{\begin{aligned}
         dX(s) &=\big[\big(A+B\Th\big)X(s)+Bv+b\big]ds \\
                 &\hp{=\ } +\big[\big(C+D\Th\big)X(s)+Dv+\si\big]dW(s),\qq\q s\in[t,T],\\
          X(t) &=x.
      \end{aligned}\right.\ee
      The set of all weak closed-loop strategies is denoted by $\cQ_w[t,T]$.
\item A weak closed-loop strategy $(\Th^*(\cd),v^*(\cd))$ is said to be {\it optimal} on $[t,T)$ if
      \bel{weak-closed-optimal*}\ba{ll}
      \ns\ds J(t,x;\Th^*(\cd)X^*(\cd)+v^*(\cd))\les J(t,x;\Th(\cd)X(\cd)+v(\cd)),\\
      \ns\ds\qq\qq\qq\qq\forall x\in\dbR^n,~\forall(\Th(\cd),v(\cd))\in\cQ_w[t,T],\ea \ee
      where $X(\cd)$ is the solution of the weak closed-loop system \rf{weak-syst}, and $X^*(\cd)$ is the solution to the weak closed-loop system \rf{weak-syst} corresponding to $(t,x)$ and $(\Th^*(\cd),v^*(\cd))$.
\item If for any $t\in[0,T)$, a weak closed-loop optimal strategy (uniquely) exists on $[t,T)$,
      we say Problem (SLQ) is ({\it uniquely}) {\it weakly closed-loop solvable}.
\end{enumerate}
\end{definition}

Similar to the case of closed-loop solvability, we have the following equivalence: A weak closed-loop strategy $(\Th^*(\cd),v^*(\cd))\in\cQ_w[t,T]$ is weakly closed-loop optimal
on $[t,T)$ if and only if
\bel{weak-closed-optimal*}J(t,x;\Th^*(\cd)X^*(\cd)+v^*(\cd))\les J(t,x;u(\cd)),\q\forall x\in\dbR^n,~\forall u(\cd)\in\cU[t,T]. \ee

\ms

We conclude this section with some existing results on the closed-loop and open-loop solvabilities of Problem (SLQ),
which play a basic role in our subsequent analysis.
For proofs and full discussion of these results, we refer the reader to Sun--Li--Yong \cite{Sun-Li-Yong 2016}.

\begin{theorem}\label{thm:SLQ-cloop-kehua}
Let {\rm\ref{ass:A1}--\ref{ass:A2}} hold. Then Problem {\rm(SLQ)} is closed-loop solvable
on $[t,T]$ if and only if the following two conditions hold:
\begin{enumerate}[\rm(i)]
\item the Riccati equation \rf{GRE} admits a solution $P(\cd)\in C([t,T];\dbS^n)$ such that
      \begin{align}
         \label{positivity}
         & R+D^\top PD\ges0, \q\ae~\hb{on}~[t,T],\\
         \label{L2-integrability}
         & \hat\Th \deq -(R+D^\top PD)^\dag(B^\top P+D^\top PC+S) \in L^2(t,T;\dbR^{m\times n}), \\
         \label{range-inclusion}
         & \sR(B^\top P+D^\top PC+S)\subseteq\sR(R+D^\top PD), \q\ae~\hb{on}~[t,T],
      \end{align}
      where $\sR(M)$ denotes the range of a matrix $M$;
\item the adapted solution $(\eta(\cd),\z(\cd))$ to the backward stochastic differential equation
      (BSDE, for short)
      \bel{eta-beta}\left\{\begin{aligned}
         d\eta(s) &=-\big[(A+B\hat\Th)^\top\eta(s) +(C+D\hat\Th)^\top\z(s) +(C+D\hat\Th)^\top P\si\\
                  &\hp{=-\big[} +\hat\Th^\top\rho +Pb +q\big]ds +\z(s) dW(s), \q s\in[t,T],\\
          \eta(T) &=g,
      \end{aligned}\right.\ee
      satisfies
      \begin{align}
      \label{SLQ:eta-z-condition-1}
      & B^\top\eta+D^\top\z+D^\top P\si+\rho \in\sR(R+D^\top PD), \q \ae~\hb{on}~[t,T],~\as, \\
      \label{SLQ:eta-z-condition-2}
      & \hat v \deq -(R+D^\top PD)^\dag(B^\top\eta+D^\top\z+D^\top P\si+\rho) \in L_\dbF^2(t,T;\dbR^m).
      \end{align}
\end{enumerate}
In this case, the closed-loop optimal strategy $(\Th^*(\cd),v^*(\cd))$ admits the following representation:
\begin{align}
\label{SLQ:barTh}
\Th^* &=\hat\Th+\big[I-(R+D^\top PD)^\dag(R+ D^\top PD)\big]\Pi,\\
\label{SLQ:barv}
v^* &=\hat v +\big[I-(R+D^\top PD)^\dag(R+ D^\top PD)\big]\pi,
\end{align}
where $(\Pi(\cd),\pi(\cd))\in L^2(t,T;\dbR^{m\times n})\times L_\dbF^2(t,T;\dbR^m)$ is arbitrary.
\end{theorem}

\begin{theorem}\label{thm:SLQ-ccloop-kehua}
Let {\rm\ref{ass:A1}--\ref{ass:A2}} hold.
\begin{enumerate}[\rm(i)]
\item Suppose Problem {\rm(SLQ)} is open-loop solvable. Then $J^0(0,0;u(\cd))\ges 0$ for all $u(\cd)\in \cU[t,T]$.
\item Suppose that there exists a constant $\l>0$ such that
      $$ J^0(0,0;u(\cd))\ges\l\dbE\int_0^T|u(s)|^2ds, \q\forall u(\cd)\in \cU[t,T]. $$
      Then the Riccati equation \rf{GRE} admits a unique solution $P(\cd)\in C([0,T];\dbS^n)$ such that
      $$ R(s)+D(s)^\top P(s)D(s)\ges \l I_m, \q\ae~s\in[0,T].$$
      Consequently, Problem {\rm(SLQ)} is uniquely closed-loop solvable and hence uniquely open-loop solvable.
      The unique closed-loop optimal strategy is given by
      $$\begin{aligned}
        \Th^* &= -\big(R+D^\top PD\big)^{-1}\big(B^\top P+D^\top PC+S\big),\\
          v^* &= -\big(R+D^\top PD\big)^{-1}\big(B^\top\eta+D^\top\z+D^\top P\si+\rho\big),
      \end{aligned}$$
      where $(\eta(\cd),\z(\cd))$ is the adapted solution to the BSDE \rf{eta-beta};
      and the unique open-loop optimal control of Problem {\rm(SLQ)} for the initial pair $(t,x)$ is given by
      $$ u^*(s) = \Th^*(s)X^*(s)+v^*(s), \q s\in[t,T],$$
      where $X^*(\cd)$ is the solution to the corresponding closed-loop system \rf{closed-syst*}.
\end{enumerate}
\end{theorem}

\section{A Perturbation Approach to Open-Loop Solvability}\label{Sec:3}

We begin by assuming
\bel{cost-convex} J^0(0,0;u(\cd))\ges0, \q\forall u(\cd)\in\cU[t,T],\ee
which, according to \autoref{thm:SLQ-ccloop-kehua} (i), is necessary for the open-loop solvability of Problem (SLQ). Condition \rf{cost-convex} means that $u(\cd)\mapsto J^0(0,0;u(\cd))$ is convex. One can actually prove that \rf{cost-convex} implies the convexity of the mapping $u(\cd)\mapsto J(t,x;u(\cd))$ for any choice of $(t,x)$ (\cite{Sun-Li-Yong 2016}).
Consider, for each $\e>0$, the LQ problem of minimizing the perturbed cost functional
\begin{eqnarray}\label{cost-e}
J_\e(t,x;u(\cd))&\3n\deq\3n& J(t,x;u(\cd))+\e\dbE\int_t^T|u(s)|^2ds, \nn\\
            &\3n=\3n& \dbE\big\{\lan GX(T),X(T)\ran + 2\lan g,X(T)\ran \nn\\
            &\3n~\3n& +\int_t^T\Bigg[\llan\begin{pmatrix}Q (s)& S(s)^\top \\ S(s) & R(s)+\e I_m\end{pmatrix}
                                          \begin{pmatrix}X(s) \\ u(s)\end{pmatrix},
                                          \begin{pmatrix}X(s) \\ u(s)\end{pmatrix}\rran \nn\\
            &\3n~\3n& +2\llan\begin{pmatrix}q(s) \\ \rho(s)\end{pmatrix},
                             \begin{pmatrix}X(s) \\ u(s)\end{pmatrix}\rran\Bigg]ds\Bigg\}
\end{eqnarray}
subject to the state equation \rf{state}.
We denote this perturbed LQ problem by Problem (SLQ)$_\e$ and its value function by $V_\e(\cd\,,\cd)$. Notice that the cost functional $J^0_\e(t,x;u(\cd))$ of the homogeneous LQ problem associated with Problem (SLQ)$_\e$ is
$$J^0_\e(t,x;u(\cd))=J^0(t,x;u(\cd))+\e\dbE\int_t^T|u(s)|^2ds,$$
which, by \rf{cost-convex}, satisfies
$$ J^0_{\e}(0,0;u(\cd))\ges\e\dbE\int_0^T|u(s)|^2ds, \q\forall u\in\cU[t,T].$$
It follows from \autoref{thm:SLQ-ccloop-kehua} that the Riccati equation
\bel{Ric-e}\left\{\begin{aligned}
  & \dot P_\e + P_\e A + A^\top P_\e + C^\top P_\e C + Q \\
  & \hp{\dot P_\e} -(P_\e B+C^\top P_\e D+S^\top)(R+\e I_m+D^\top P_\e D)^{-1}(B^\top P_\e +D^\top P_\e C+S)=0, \\
  & P_\e(T)=G
\end{aligned}\right.\ee
associated with Problem (SLQ)$_\e$ has a unique solution $P_\e(\cd)\in C([0,T];\dbS^n)$ such that
$$ R(s)+\e I_m+D(s)^\top P_\e(s)D(s) \ges \e I_m, \q\ae~s\in[0,T]. $$
Furthermore, let $(\eta_\e(\cd),\z_\e(\cd))$ be the adapted solution to the BSDE 
\bel{eta-zeta-e}\left\{\begin{aligned}
  d\eta_\e(s) &=-\big[(A+B\Th_\e)^\top\eta_\e(s)+(C+D\Th_\e)^\top\z_\e(s) + (C+D\Th_\e)^\top P_\e\si \\
              &\hp{=-\big[} +\Th_\e^\top\rho + P_\e b +q\big]ds + \z_\e(s) dW(s), \q s\in[0,T],\\
   \eta_\e(T) &=g,
\end{aligned}\right.\ee
and let $X_\e(\cd)$ be the solution to the closed-loop system

\bel{Xe}\left\{\begin{aligned}
   dX_\e(s) &= [(A+B\Th_\e)X_\e(s)+Bv_\e+b]ds \\
            &\hp{=\ } +[(C+D\Th_\e)X_\e(s)+Dv_\e+\si]dW(s), \q s\in[t,T],\\
    X_\e(t) &= x,
\end{aligned}\right.\ee
where $\Th_\e(\cd):[0,T]\to\dbR^{m\times n}$ and $v_\e(\cd):[0,T]\times\Om\to\dbR^m$ are defined by
\begin{align}
\label{Th-e}
\Th_\e &= -(R+\e I_m+D^\top P_\e D)^{-1}(B^\top P_\e+D^\top P_\e C+S), \\
\label{v-e}
v_\e &= -(R+\e I_m+D^\top P_\e D)^{-1}(B^\top\eta_\e+D^\top\z_\e+D^\top P_\e\si+\rho).
\end{align}
Then the unique open-loop optimal control of Problem (SLQ)$_\e$ for the initial pair $(t,x)$ is given by
\bel{u-e} u_\e(s) = \Th_\e(s)X_\e(s) + v_\e(s), \q s\in[t,T]. \ee

\ss

Now we present the main result of this section, which provides a characterization of the open-loop
solvability of Problem (SLQ) in terms of the family $\{u_\e(\cd)\}_{\e>0}$.

\begin{theorem}\label{thm:SLQ-oloop-kehua}
Let {\rm\ref{ass:A1}--\ref{ass:A2}} and \rf{cost-convex} hold.
For any given initial pair $(t,x)\in[0,T)\times\dbR^n$, let $u_\e(\cd)$ be defined by \rf{u-e}, which is the outcome of the closed-loop optimal strategy $(\Th_\e(\cd),v_\e(\cd))$ of Problem {\rm(SLQ)}$_\e$. Then the following statements are equivalent:
\begin{enumerate}[\rm(i)]
\item Problem {\rm(SLQ)} is open-loop solvable at $(t,x)$;
\item the family $\{u_\e(\cd)\}_{\e>0}$ is bounded in the Hilbert space $L^2_\dbF(t,T;\dbR^m)$, i.e.,
      $$\sup_{\e>0}\,\dbE\int_t^T|u_\e(s)|^2ds <\i;$$
\item the family $\{u_\e(\cd)\}_{\e>0}$ is convergent strongly in $L^2_\dbF(t,T;\dbR^m)$ as $\e\to0$.
\end{enumerate}
Whenever {\rm(i)}, {\rm(ii)}, or {\rm(iii)} is satisifed, the family $\{u_\e(\cd)\}_{\e> 0}$ converges strongly to an open-loop
optimal control of Problem {\rm(SLQ)} for the initial pair $(t,x)$ as $\e\to0$.
\end{theorem}

To prove \autoref{thm:SLQ-oloop-kehua}, we need the following lemma.

\begin{lemma}
Let {\rm\ref{ass:A1}--\ref{ass:A2}} hold. Then for any initial pair $(t,x)\in[0,T)\times\dbR^n$,
\bel{Ve-goto-V} \lim_{\e\downarrow0}V_\e(t,x)=V(t,x). \ee
\end{lemma}

\begin{proof}
Let $(t,x)\in[0,T)\times\dbR^n$ be fixed. For any $\e>0$ and any $u(\cd)\in\cU[t,T]$, we have
$$ J_\e(t,x;u(\cd))=J(t,x;u(\cd))+\e\dbE\int_t^T|u(s)|^2ds \ges J(t,x;u(\cd))\ges V(t,x).$$
Taking the infimum over all $u(\cd)\in\cU[t,T]$ on the left-hand side gives
\bel{Ve<V-case1} V_\e(t,x) \ges V(t,x). \ee
On the other hand, if $V(t,x)$ is finite, then for any $\d>0$, we can find a $u^\d(\cd)\in\cU[t,T]$, independent of $\e>0$, such that $J(t,x;u^\d(\cd))\les V(t,x)+\d$. It follows that
$$ V_\e(t,x)\les J(t,x;u^\d(\cd))+\e\dbE\int_t^T|u^\d(s)|^2ds\les V(t,x)+\d+\e\dbE\int_t^T|u^\d(s)|^2ds. $$
Letting $\e\to0$, we obtain
\bel{Ve<V-case2} V_\e(t,x) \les V(t,x)+\d. \ee
Since $\d>0$ is arbitrary, we obtain \rf{Ve-goto-V} by combining \rf{Ve<V-case1} and \rf{Ve<V-case2}.
A similar argument applies to the case when $V(t,x)=-\i$.
\end{proof}

\no\it Proof of \autoref{thm:SLQ-oloop-kehua}. \rm
We begin by proving the implication  (i) $\Rightarrow$ (ii).
Let $v^*(\cd)$ be an open-loop optimal control of Problem (SLQ) for the initial pair $(t,x)$. Then for any $\e>0$,
\begin{align}\label{Ve<V+}
  V_\e(t,x) &\les J_\e(t,x;v^*(\cd))=J(t,x;v^*(\cd))+\e\dbE\int_t^T|v^*(s)|^2ds \nn\\
            &= V(t,x) + \e\dbE\int_t^T|v^*(s)|^2ds.
\end{align}
On the other hand, since $u_\e(\cd)$ is optimal for Problem (SLQ)$_\e$ with respect to $(t,x)$, we have
\begin{align}\label{Ve>V+}
  V_\e(t,x) &= J_\e(t,x;u_\e(\cd)) = J(t,x;u_\e(\cd)) + \e\dbE\int_t^T|u_\e(s)|^2ds \nn\\
            &\ges V(t,x) +\e\dbE\int_t^T|u_\e(s)|^2ds.
\end{align}
Combining \rf{Ve>V+} and \rf{Ve<V+} yields
\bel{bound-of-ue} \dbE\int_t^T|u_\e(s)|^2ds \les {V_\e(t,x)-V(t,x)\over\e} \les \dbE\int_t^T|v^*(s)|^2ds. \ee
This shows that $\{u_\e(\cd)\}_{\e>0}$ is bounded in $L^2_\dbF(t,T;\dbR^m)$.

\ms

We next show that (ii) $\Rightarrow$ (i).
Since $\{u_\e(\cd)\}_{\e>0}\subseteq L^2_\dbF(t,T;\dbR^m)$ is bounded, we can extract a sequence
$\{\e_k\}^\i_{k=1}\subseteq(0,\i)$ with $\lim_{k\to\i}\e_k=0$ such that $\{u_{\e_k}(\cd)\}$
converges weakly to some $u^*(\cd)\in L^2_\dbF(t,T;\dbR^m)$.
Note that the mapping $u(\cd)\mapsto J(t,x;u(\cd))$ is sequentially weakly lower semicontinuous because it is continuous and convex.
Then the boundedness of $\{u_{\e_k}(\cd)\}$, together with \rf{Ve-goto-V}, implies that
\begin{align*}
J(t,x;u^*(\cd))&\les \liminf_{k\to\i} J(t,x;u_{\e_k}(\cd)) \\
          &= \liminf_{k\to\i}\lt[ V_{\e_k}(t,x)-\e_k\dbE\int_t^T|u_{\e_k}(s)|^2ds \rt]= V(t,x).
\end{align*}
This means that $u^*(\cd)$ is an open-loop optimal control of Problem (SLQ) for $(t,x)$.

\ms

The implication (iii) $\Rightarrow$ (ii) is trivially true.

\ms

Finally, we prove the implication (ii) $\Rightarrow$ (iii). The proof is divided into two steps.

\ms

{\it Step 1:  The family $\{u_\e(\cd)\}_{\e>0}$ converges weakly to an open-loop optimal control of Problem {\rm(SLQ)} for the initial pair $(t,x)$ as $\e\to0$.}

\ms

To verify this, it suffices to show that every weakly convergent subsequence of $\{u_\e(\cd)\}_{\e>0}$ has the same weak limit which is an open-loop optimal control of Problem {\rm(SLQ)} for $(t,x)$.
Let $u_i^*(\cd)$; $i=1,2$, be the weak limits of two different weakly convergent subsequences $\{u_{i,\e_k}(\cd)\}_{k=1}^\infty$ $(i=1,2)$ of $\{u_\e(\cd)\}_{\e>0}$. The same argument as in the proof of (ii) $\Rightarrow$ (i) shows that both $u_1^*(\cd)$ and $u_2^*(\cd)$ are
optimal for $(t,x)$. Thus, recalling that the mapping $u(\cd)\mapsto J(t,x;u(\cd))$ is convex, we have
$$ J\lt(t,x;{u_1^*(\cd)+u_2^*(\cd)\over 2}\rt)\les {1\over2}J(t,x;u_1^*(\cd)) +{1\over2}J(t,x;u_2^*(\cd))= V(t,x). $$
This means that ${u_1^*(\cd)+u_2^*(\cd)\over 2}$ is also optimal for Problem (SLQ) with respect to $(t,x)$. Then we can repeat the argument employed in the proof of (i) $\Rightarrow$ (ii), replacing $v^*(\cd)$ by ${u_1^*(\cd)+u_2^*(\cd)\over 2}$, to obtain (see \rf{bound-of-ue})
$$ \dbE\int_t^T|u_{i,\e_k}(s)|^2ds \les \dbE\int_t^T\lt|{u_1^*(s)+u_2^*(s)\over 2}\rt|^2ds, \q i=1,2. $$
Taking inferior limits then yields
$$ \dbE\int_t^T|u_i^*(s)|^2ds \les \dbE\int_t^T\lt|{u_1^*(s)+u_2^*(s)\over 2}\rt|^2ds, \q i=1,2. $$
Adding the above two inequalities and then multiplying by $2$, we get
$$ 2\lt[\dbE\int_t^T|u_1^*(s)|^2ds+\dbE\int_t^T|u_2^*(s)|^2ds\rt] \les \dbE\int_t^T|u_1^*(s)+u_2^*(s)|^2ds, $$
or equivalently (by shifting the integral on the right-hand side to the left-hand side),
$$ \dbE\int_t^T|u_1^*(s)-u_2^*(s)|^2ds\les 0. $$
It follows that $u_1^*(\cd)=u_2^*(\cd)$, which establishes the claim.

\ms

{\it Step 2: The family $\{u_\e(\cd)\}_{\e>0}$ converges strongly as $\e\to0$. }

\ms

According to Step 1, the family $\{u_\e(\cd)\}_{\e>0}$ converges weakly to an open-loop optimal control $u^*(\cd)$ of Problem {\rm(SLQ)} for $(t,x)$ as $\e\to0$.
By repeating the argument employed in the proof of (i) $\Rightarrow$ (ii) with $u^*(\cd)$ replacing $v^*(\cd)$, we obtain
\bel{t-3-5}\dbE\int_t^T|u_\e(s)|^2ds \les \dbE\int_t^T|u^*(s)|^2ds, \q\forall\e>0.\ee
On the other hand, since $u^*(\cd)$ is the weak limit of $\{u_\e(\cd)\}_{\e>0}$, we have
$$ \dbE\int_t^T|u^*(s)|^2ds \les \liminf_{\e\to0} \dbE\int_t^T|u_\e(s)|^2ds.$$
Combining this with \rf{t-3-5}, we see that $\dbE\int_t^T|u_\e(s)|^2ds$ actually has the limit $\dbE\int_t^T|u^*(s)|^2ds$. Thus (recalling that $\{u_\e(\cd)\}_{\e>0}$ converges weakly to $u^*(\cd)$),
\begin{align*}
  &\lim_{\e\to0}\dbE\int_t^T|u_\e(s)-u^*(s)|^2ds \\
  &\q =\lim_{\e\to0}\lt[\dbE\int_t^T|u_\e(s)|^2ds + \dbE\int_t^T|u^*(s)|^2ds -2\,\dbE\int_t^T\lan u^*(s),u_\e(s)\ran ds\rt]\\
  &\q=0,
\end{align*}
which means that $\{u_\e(\cd)\}_{\e>0}$ converges strongly to $u^*(\cd)$ as $\e\to0$. $\hfill\qed$

\begin{remark}\rm
A similar result first appeared in \cite{Sun-Li-Yong 2016}, which asserts that if Problem (SLQ) is open-loop solvable at $(t,x)$, then the limit of any weakly/strongly convergent subsequence of $\{u_\e(\cd)\}_{\e>0}$ is an
open-loop optimal control for $(t,x)$. Our result sharpens that in \cite{Sun-Li-Yong 2016} by showing the family $\{u_\e(\cd)\}_{\e>0}$ itself is strongly convergent when Problem (SLQ) is open-loop solvable. This improvement has at least two advantages. First, it serves as a crucial bridge to the weak closed-loop solvability presented in the next section.
Second, it is much more convenient for computational purposes because subsequence extraction is not required.
\end{remark}

\section{Weak Closed-Loop Solvability}\label{Sec:4}

In this section, we establish the equivalence between open-loop and weak closed-loop solvabilities of Problem (SLQ). We shall show that $\Th_\e(\cd)$ and $v_\e(\cd)$ defined by \rf{Th-e} and \rf{v-e} converge locally in $[0,T)$, and that the limit pair $(\Th^*(\cd),v^*(\cd))$ is a weak closed-loop optimal strategy. We emphasize the fact that in general the limits $\Th^*(\cd)$ and $v^*(\cd)$ are merely locally square-integrable
over $[0,T)$ and cannot be obtained directly by solving the associated Riccati equation and BSDE (see Examples \ref{ex-1.1} and \ref{ex-5.1}).

\ms

We start with a simple lemma, which will enable us to work separately with $\Th_\e(\cd)$ and $v_\e(\cd)$. Recall that the associated Problem (SLQ)$^0$ is to minimize \rf{cost0} subject to \rf{state0}.

\begin{lemma}\label{lmm:SLQ-SLQ0}
Let {\rm\ref{ass:A1}--\ref{ass:A2}} hold.
If Problem {\rm(SLQ)} is open-loop solvable, then so is Problem {\rm(SLQ)$^0$}.
\end{lemma}

\begin{proof}
Let $(t,x)\in[0,T)\times\dbR^n$ be arbitrary. We note first that if $b(\cd),\si(\cd),g,q(\cd),\rho(\cd)=0$, then the adapted solution $(\eta_\e(\cd),\z_\e(\cd))$ to BSDE \rf{eta-zeta-e} is identically $(0,0)$
and hence the process $v_\e(\cd)$ defined by \rf{v-e} is identically zero.
So by \autoref{thm:SLQ-oloop-kehua}, to prove that Problem (SLQ)$^0$ is open-loop solvable at $(t,x)$ we need to verify that the family $\{u_\e(\cd)\}_{\e>0}$ is bounded in $L^2_\dbF(t,T;\dbR^m)$, with $u_\e(\cd)=\Th_\e(\cd)X_\e(\cd)$, where $X_\e(\cd)$ is the solution to the following:
\bel{Xe-5-30}\left\{\begin{aligned}
   dX_\e(s) &= (A+B\Th_\e)X_\e(s)ds + (C+D\Th_\e)X_\e(s) dW(s),\q s\in[t,T],\\
    X_\e(t) &= x.
\end{aligned}\right.\ee
To this end, we return to Problem (SLQ).
Let $v_\e(\cd)$ be defined in \rf{v-e} and denote by $X_\e^{t,x}(\cd)$ and $X_\e^{t,0}(\cd)$ the solutions to \rf{Xe}
with respect to the initial pairs $(t,x)$ and $(t,0)$, respectively.
Since Problem (SLQ) is open-loop solvable at both $(t,x)$ and $(t,0)$, by \autoref{thm:SLQ-oloop-kehua},
the families
$$ u_\e^{t,x}(\cd)\deq \Th_\e(\cd)X_\e^{t,x}(\cd)+ v_\e(\cd) \q\hb{and}\q u_\e^{t,0}(\cd) \deq \Th_\e(\cd) X_\e^{t,0}(\cd) + v_\e(\cd)$$
are bounded in $L^2_\dbF(t,T;\dbR^m)$. Note that because the process $v_\e(\cd)$ is independent of the initial state, the difference $X_\e^{t,x}(\cd)-X_\e^{t,0}(\cd)$ satisfies the same SDE as $X_\e(\cd)$.
By the uniqueness of solutions of SDEs, we must have
$$X_\e(\cd)= X_\e^{t,x}(\cd)-X_\e^{t,0}(\cd).$$
It follows that $u_\e(\cd)=u_\e^{t,x}(\cd)-u_\e^{t,0}(\cd)$. Because $\{u_\e^{t,x}(\cd)\}_{\e>0}$ and $\{u_\e^{t,0}(\cd)\}_{\e>0}$
are bounded in $L^2_\dbF(t,T;\dbR^m)$, so is $\{u_\e(\cd)\}_{\e>0}$. This implies that Problem (SLQ)$^0$ is open-loop solvable. \end{proof}

We now prove that the family $\{\Th_\e(\cd)\}_{\e>0}$ defined by \rf{Th-e} is locally convergent in $[0,T)$.

\begin{proposition}\label{prop:limit-The} Let {\rm\ref{ass:A1}--\ref{ass:A2}} hold. Suppose that Problem {\rm(SLQ)$^0$} is open-loop solvable.
Then the family $\{\Th_\e(\cd)\}_{\e>0}$ defined by \rf{Th-e} converges in $L^2(0,T';\dbR^{m\times n})$ for any $0<T'<T$; that is, there exists a locally square-integrable deterministic function $\Th^*(\cdot):[0,T)\to\dbR^{m\times n}$ such that
$$ \lim_{\e\to 0}\int_0^{T'}|\Th_\e(s)-\Th^*(s)|^2ds=0, \q\forall\, 0<T'<T. $$
\end{proposition}

\begin{proof}
We need to show that for any $0<T'<T$, the family $\{\Th_\e(\cd)\}_{\e>0}$ is Cauchy in $L^2(0,T';\dbR^{m\times n})$. To this end, let us first fix an arbitrary initial time $t\in[0,T)$ and let $\F_\e(\cd)\in L_\dbF^2(\Om;C([t,T];\dbR^{n\times n})$ be the solution to the following SDE for $\dbR^{n\times n}$-valued processes:
\bel{Phi-e}\left\{\begin{aligned}
   d\Phi_\e(s) &= (A+B\Th_\e)\Phi_\e(s)ds +(C+D\Th_\e)\Phi_\e(s)dW(s), \q s\in[t,T],\\
    \Phi_\e(t) &= I_n.
\end{aligned}\right.\ee
%
Clearly, for any initial state $x$, the solution of \rf{Xe-5-30} is given by
$$ X_\e(s) = \F_\e(s)x, \q s\in[t,T]. $$
Since Problem (SLQ)$^0$ is open-loop solvable, by \autoref{thm:SLQ-oloop-kehua}, the family
$$u_\e(s) =\Th_\e(s)X_\e(s) = \Th_\e(s)\F_\e(s)x,\q s\in[t,T]; \q \e>0$$
is strongly convergent in $L^2_\dbF(t,T;\dbR^m)$ for any $x\in\dbR^n$.
It follows that $\{\Th_\e(\cd)\F_\e(\cd)\}_{\e>0}$ converges strongly in $L^2_\dbF(t,T;\dbR^{m\times n})$ as $\e\to0$. Denote $U_\e(\cd)=\Th_\e(\cd)\F_\e(\cd)$ and let $U^*(\cd)$ be the strong limit of $U_\e(\cd)$. One sees that $\dbE[\F_\e(\cd)]$ satisfies the following ODE:
$$\left\{\begin{aligned}
    d\dbE[\Phi_\e(s)]  &= \{A(s)\dbE[\Phi_\e(s)] + B(s)\dbE[U_\e(s)]\}ds, \q s\in[t,T], \\
       \dbE[\Phi_\e(t)] &= I_n,
\end{aligned}\right.$$
and the H\"{o}lder inequality implies
$$ \int_t^T|\dbE[U_\e(s)]-\dbE[U^*(s)]|^2 ds \les \dbE\int_t^T|U_\e(s)-U^*(s)|^2 ds \to 0 \q\hb{as}\q \e\to0. $$
By the standard results of ODE, the family of continuous functions $\dbE[\F_\e(\cd)]$  converges uniformly to the solution  of
$$\left\{\begin{aligned}
    d\dbE[\Phi^*(s)]  &= \{A(s)\dbE[\Phi^*(s)] + B(s)\dbE[U^*(s)]\}ds, \q s\in[t,T], \\
       \dbE[\Phi^*(t)] &= I_n.
\end{aligned}\right.$$
Thus, by noting that $\dbE[\Phi^*(t)]=I_n$ we can choose a small constant $\D_t>0$ such that for small $\e>0$,
\begin{enumerate}[(a)]
\item $\dbE[\F_\e(s)]$ is invertible for all $s\in[t,t+\D_t]$,  and
\item $|\dbE[\F_\e(s)]|\ges {1\over2}$ for all $s\in[t,t+\D_t]$.
\end{enumerate}
We claim that the family $\{\Th_\e(\cd)\}_{\e>0}$ is Cauchy in $L^2(t,t+\D_t;\dbR^{m\times n})$. Indeed, by (a) and (b), we have
\begin{align*}
& \int_t^{t+\D_t} |\Th_{\e_1}(s)-\Th_{\e_2}(s)|^2 ds \\
&\q= \int_t^{t+\D_t} \lt|\dbE[U_{\e_1}(s)]\dbE[\Phi_{\e_1}(s)]^{-1}-\dbE[U_{\e_2}(s)]\dbE[\Phi_{\e_2}(s)]^{-1}\rt|^2ds \\
&\q\les 2\int_t^{t+\D_t} \big|\dbE[U_{\e_1}(s)-U_{\e_2}(s)]\big|^2  \big|\dbE[\Phi_{\e_1}(s)]^{-1}\big|^2 ds \\
&\q\hp{\les\ } +2\int_t^{t+\D_t} \big|\dbE[U_{\e_2}(s)]\big|^2 \big|\dbE[\Phi_{\e_1}(s)]^{-1}-\dbE[\Phi_{\e_2}(s)]^{-1}\big|^2ds  \\
&\q= 2\int_t^{t+\D_t} \big|\dbE[U_{\e_1}(s)-U_{\e_2}(s)]\big|^2  \big|\dbE[\Phi_{\e_1}(s)]^{-1}\big|^2 ds \\
&\q\hp{\les\ } +2\int_t^{t+\D_t} \big|\dbE[U_{\e_2}(s)]\big|^2 \big|\dbE[\Phi_{\e_1}(s)]^{-1}\big|^2 \big|\dbE[\Phi_{\e_2}(s)]-\dbE[\Phi_{\e_1}(s)]\big|^2 \big|\dbE[\Phi_{\e_2}(s)]^{-1}\big|^2ds  \\
&\q\les 8\int_t^{t+\D_t} \big|\dbE[U_{\e_1}(s)-U_{\e_2}(s)]\big|^2 ds \\
&\q\hp{\les\ } +32\int_t^{t+\D_t} \big|\dbE[U_{\e_2}(s)]\big|^2 ds  \cd \sup_{t\les s\les t+\D_t}\big|\dbE[\Phi_{\e_1}(s)]-\dbE[\Phi_{\e_2}(s)]\big|^2.
\end{align*}
Since $\{U_\e(\cd)\}_{\e>0}$ is Cauchy in $L^2_\dbF(t,T;\dbR^{m\times n})$ and $\{\dbE[\F_\e(\cd)]\}_{\e>0}$ converges uniformly on $[t,T]$,
the last two terms of the above inequality approach to zero as $\e_1,\e_2\to0$.

\ms

Next we use a compactness argument to prove that $\{\Th_\e(\cd)\}_{\e>0}$ is actually Cauchy in $L^2(0,T';\dbR^{m\times n})$
for any $0<T'<T$. Take any $T'\in(0,T)$. From the preceding argument we see that for each $t\in[0,T']$, there exists a small $\D_t>0$ such that $\{\Th_\e(\cd)\}_{\e>0}$ is Cauchy in $L^2(t,t+\D_t;\dbR^{m\times n})$. Since $[0,T']$ is compact, we can choose finitely many $t\in[0,T']$, say, $t_1,t_2,\ldots,t_k$, such that
$\{\Th_\e(\cd)\}_{\e>0}$ is Cauchy in each $L^2(t_j,t_j+\D_{t_j};\dbR^{m\times n})$ and $[0,T']\subseteq\bigcup_{j=1}^k[t_j,t_j+\D_{t_j}]$. It follows that
$$ \int_0^{T^\prime} |\Th_{\e_1}(s)-\Th_{\e_2}(s)|^2 ds
\les \sum_{j=1}^k\int_{t_j}^{t_j+\D_{t_j}} |\Th_{\e_1}(s)-\Th_{\e_2}(s)|^2 ds \to0 \q\hb{as}\q \e_1,\e_2\to0. $$
The proof is therefore completed.
\end{proof}

The next result shows that the family $\{v_\e(\cd)\}_{\e>0}$ defined by \rf{v-e} is also locally convergent in $[0,T)$.

\begin{proposition}\label{prop:limit-ve}
Let {\rm\ref{ass:A1}--\ref{ass:A2}} hold. Suppose that Problem {\rm(SLQ)} is open-loop solvable. Then the family $\{v_\e(\cd)\}_{\e>0}$ defined by \rf{v-e} converges in $L_\dbF^2(0,T^\prime;\dbR^m)$
for any $0<T^\prime<T$; that is, there exists a locally square-integrable process
$v^*(\cdot):[0,T)\times\Om\to\dbR^m$ such that
$$ \lim_{\e\to0}\dbE\int_0^{T^\prime}|v_\e(s)-v^*(s)|^2ds=0, \q\forall\, 0<T^\prime<T. $$
\end{proposition}

\begin{proof}
Let $X_\e(s)$; $0\les s\les T$ be the solution to the closed-loop system \rf{Xe} with respect to initial time $t=0$.
Since Problem (SLQ) is open-loop solvable, \autoref{thm:SLQ-oloop-kehua} implies that the family
$$ u_\e(s) = \Th_\e(s)X_\e(s) + v_\e(s), \q s\in[0,T]; \q\e>0 $$
is Cauchy in $L^2_\dbF(0,T;\dbR^m)$; that is,
$$\dbE\int_0^T|u_{\e_1}(s)-u_{\e_2}(s)|^2ds \to 0  \q\hb{as}\q  \e_1,\e_2 \to0. $$
It follows by the linearity of the state equation \rf{state} and \autoref{lmm:well-posedness-SDE} that
\bel{X1-X2=go0} \dbE\lt[\sup_{0\les s\les T}|X_{\e_1}(s)-X_{\e_2}(s)|^2\rt]
\les K\dbE\int^T_0|u_{\e_1}(s)-u_{\e_2}(s)|^2ds \to0  \q\hb{as}\q  \e_1,\e_2 \to0. \ee
Now take any $0<T^\prime<T$. Since Problem (SLQ) is open-loop solvable, according to \autoref{lmm:SLQ-SLQ0} and \autoref{prop:limit-The}, the family $\{\Th_\e(\cd)\}_{\e>0}$ is Cauchy in $L^2(0,T^\prime;\dbR^{m\times n})$.
Thus, making use of \rf{X1-X2=go0}, we obtain
\begin{align*}
& \dbE\int_0^{T^\prime}|\Th_{\e_1}(s)X_{\e_1}(s)-\Th_{\e_2}(s)X_{\e_2}(s)|^2ds \\
&\q\les 2\dbE\int_0^{T^\prime}|\Th_{\e_1}(s)-\Th_{\e_2}(s)|^2|X_{\e_1}(s)|^2ds
        + 2\dbE\int_0^{T^\prime}|\Th_{\e_2}(s)|^2|X_{\e_1}(s)-X_{\e_2}(s)|^2 ds \\
&\q\les 2\int_0^{T^\prime}|\Th_{\e_1}(s)-\Th_{\e_2}(s)|^2ds\cd \dbE\left[ \sup_{0\les s\les T^\prime}|X_{\e_1}(s)|^2\right] \\
&\q\hp{\les\ } + 2\int_0^{T^\prime}|\Th_{\e_2}(s)|^2ds\cd\dbE\left[\sup_{0\les s\les T^\prime}|X_{\e_1}(s)-X_{\e_2}(s)|^2\right] \\
&\q\to 0  \q\hb{as}\q  \e_1,\e_2 \to0,
\end{align*}
and therefore
\begin{align*}
& \dbE\int_0^{T^\prime}|v_{\e_1}(s)-v_{\e_2}(s)|^2ds \\
&\q=\dbE\int_0^{T^\prime}\big|[u_{\e_1}(s)-\Th_{\e_1}(s)X_{\e_1}(s)]-[u_{\e_2}(s)-\Th_{\e_2}(s)X_{\e_2}(s)]\big|^2ds \\
&\q\les 2\dbE\int_0^{T^\prime}|u_{\e_1}(s)-u_{\e_2}(s)|^2 + 2\dbE\int_0^{T^\prime}|\Th_{\e_1}(s)X_{\e_1}(s)-\Th_{\e_2}(s)X_{\e_2}(s)|^2ds \\
&\q\to 0  \q\hb{as}\q  \e_1,\e_2 \to0.
\end{align*}
This shows that the family $\{v_\e(\cd)\}_{\e>0}$ converges in $L_\dbF^2(0,T^\prime;\dbR^m)$.
\end{proof}

We are now ready to state and prove the main result of this section, which establishes the equivalence between open-loop and weak closed-loop solvabilities of Problem (SLQ).

\begin{theorem}\label{thm:open=weak-closed}
Let {\rm\ref{ass:A1}--\ref{ass:A2}} hold. If Problem {\rm(SLQ)} is open-loop solvable, then the limit pair $(\Th^*(\cd),v^*(\cd))$ obtained in Propositions \ref{prop:limit-The} and \ref{prop:limit-ve} is a weak closed-loop optimal strategy of Problem {\rm(SLQ)} on any $[t,T)$. Consequently, the open-loop and weak closed-loop solvabilities of Problem {\rm(SLQ)} are equivalent.
\end{theorem}

\begin{proof}
Take an arbitrary initial pair $(t,x)\in[0,T)\times\dbR^n$ and let $\{u_\e(s);t\les s\les T\}_{\e>0}$
be the family defined by \rf{u-e}.  Since Problem (SLQ) is open-loop solvable at $(t,x)$,
by \autoref{thm:SLQ-oloop-kehua}, $\{u_\e(s);t\les s\les T\}_{\e>0}$ converges strongly to an open-loop
optimal control $\{u^*(s);t\les s\les T\}$ of Problem (SLQ) (for the initial pair $(t,x)$).
Let $\{X^*(s);t\les s\les T\}$ be the corresponding optimal state process; that is, $X^*$ is the solution to
$$\left\{\begin{aligned}
   dX^*(s) &=[A(s)X^*(s)+ B(s)u^*(s)+ b(s)]ds\\
         &\hp{=\ } +[C(s)X^*(s) + D(s)u^*(s)+\si(s)]dW(s),\q s\in[t,T],\\
     X^*(t)&= x.
\end{aligned}\right.$$
If we can show that
\bel{18-6-1}u^*(s)=\Th^*(s)X^*(s)+v^*(s), \q t\les s<T,\ee
then $(\Th^*(\cd),v^*(\cd))$ is clearly a weak closed-loop optimal strategy of Problem {\rm(SLQ)} on $[t,T)$.
To justify the argument, we note first that by \autoref{lmm:well-posedness-SDE},
$$\dbE\lt[\sup_{t\les s\les T}|X_\e(s)-X^*(s)|^2\rt] \les K\dbE\int^T_t|u_\e(s)-u^*(s)|^2ds \to0  \q\hb{as}\q  \e\to0,$$
where $\{X_\e(s);t\les s\les T\}$ is the solution to equation \rf{Xe}.
Second, by Propositions \ref{prop:limit-The} and \ref{prop:limit-ve},
\begin{align*}
\lim_{\e\to0}\int_0^{T^\prime}|\Th_\e(s)-\Th^*(s)|^2ds=0, \q\forall\, 0<T^\prime<T, \\
\lim_{\e\to0}\dbE\int_0^{T^\prime}|v_\e(s)-v^*(s)|^2ds=0, \q\forall\, 0<T^\prime<T.
\end{align*}
It follows that for any $0<T^\prime<T$,
\begin{align*}
&\dbE\int_0^{T^\prime} \big|[\Th_\e(s)X_\e(s)+v_\e(s)]-[\Th^*(s)X^*(s)+v^*(s)]\big|^2ds \\
&\q\les 2\dbE\int_0^{T^\prime} |v_\e(s)-v^*(s)|^2ds +2\dbE\int_0^{T^\prime} |\Th_\e(s)X_\e(s)-\Th^*(s)X^*(s)|^2ds \\
&\q\les 2\dbE\int_0^{T^\prime} |v_\e(s)-v^*(s)|^2ds +4\dbE\int_0^{T^\prime} |\Th_\e(s)|^2|X_\e(s)-X^*(s)|^2ds \\
&\q\hp{\les\ } + 4\dbE\int_0^{T^\prime} |\Th_\e(s)-\Th^*(s)|^2|X^*(s)|^2ds \\
&\q\les 2\dbE\int_0^{T'} |v_\e(s)-v^*(s)|^2ds
        +4\int_0^{T'}|\Th_\e(s)|^2ds \cd\dbE\[\sup_{t\les s\les T}|X_\e(s)-X^*(s)|^2\] \\
&\q\hp{\les\ } +4\int_0^{T^\prime} |\Th_\e(s)-\Th^*(s)|^2ds \cd\dbE\[\sup_{t\les s\les T}|X^*(s)|^2\]\to 0 \q\hb{as}\q \e\to0.
\end{align*}
Recall that $u_\e(s)=\Th_\e(s)X_\e(s)+v_\e(s);t\les s\les T$ converges strongly to $u^*(s);t\les s\les T$ in $L^2_\dbF(t,T;\dbR^m)$ as $\e\to0$.
Thus, \rf{18-6-1} must hold.
The above argument shows that the open-loop solvability implies the weak closed-loop solvability.
The reverse implication is obvious by \autoref{def-wcloop}.
\end{proof}

\section{An Example}\label{Sec:Example}

In this section we present an example in which the LQ problem is open-loop solvable
(and hence weakly closed-loop solvable) but not closed-loop solvable.
This example illustrates the procedure for finding weak closed-loop optimal strategies.

\begin{example}\label{ex-5.1}\rm
Consider the following Problem (SLQ) with one-dimensional state equation
$$\left\{\begin{aligned}
   dX(s) &= [-X(s)+u(s)+b(s)]ds + \sqrt{2}X(s)dW(s), \q s\in[t,1],\\
    X(t) &= x,
\end{aligned}\right.$$
and cost functional
$$ J(t,x;u(\cd))=\dbE |X(1)|^2, $$
where the nonhomogeneous term $b(\cd)$ is given by
$$ b(s)=\left\{\begin{aligned}
   & {e^{\sqrt{2}W(s)-2s}\over\sqrt{1-s}}; && s\in[0,1), \\
   & 0;                                    && s=1.
\end{aligned}\right.$$
It is easily seen that $b(\cd)\in L^2_\dbF(\Om;L^1(0,1;\dbR))$. In fact,
\begin{align*}
\dbE\lt(\int_0^1|b(s)|ds\rt)^2
   &= \dbE\lt(\int_0^1{e^{\sqrt{2}W(s)-2s}\over\sqrt{1-s}} ds\rt)^2
      \les \dbE\lt(\int_0^1{e^{\sqrt{2}W(s)-s}\over\sqrt{1-s}} ds\rt)^2 \\
   &\les \dbE\lt(\int_0^1{1\over\sqrt{1-s}} ds \cd\sup_{0\les s\les1}e^{\sqrt{2}W(s)-s}\rt)^2 \\
   &= \lt(\int_0^1{1\over\sqrt{1-s}} ds\rt)^2 \cd\dbE\lt(\sup_{0\les s\les1}e^{\sqrt{2}W(s)-s}\rt)^2 \\
   &= 4\,\dbE\lt(\sup_{0\les s\les1}e^{\sqrt{2}W(s)-s}\rt)^2.
\end{align*}
Since $\{e^{\sqrt{2}W(s)-s};s\ges0\}$ is a square-integrable martingale,
it follows from Doob's maximal inequality that
$$\dbE\lt(\sup_{0\les s\les1}e^{\sqrt{2}W(s)-s}\rt)^2 \les 4\,\dbE e^{2\sqrt{2}W(1)-2} = 4e^2.$$
Thus, $\dbE\big(\int_0^1|b(s)|ds\big)^2 \les16e^2 <\i$.

\ms

We first claim that this LQ problem is not closed-loop solvable on any $[t,1]$.
Indeed, the generalized Riccati equation associated with this problem reads
$$\left\{\begin{aligned}
   \dot P(s) &=P(s)0^\dag P(s)=0, \q s\in[t,1],\\
        P(1) &=1,
\end{aligned}\right.$$
whose solution is, obviously, $P(s)\equiv 1$. For any $s\in[t,1]$, we have
\begin{align*}
   &\sR(B(s)^\top P(s)+D(s)^\top P(s)C(s)+S(s))=\sR(1)=\dbR, \\
   &\sR(R(s)+D(s)^\top P(s)D(s))=\sR(0)=\{0\},
\end{align*}
so the range inclusion condition \rf{range-inclusion} is not satisfied.
Our claim then follows from \autoref{thm:SLQ-cloop-kehua}.

\ms

Next we use \autoref{thm:SLQ-oloop-kehua} to conclude that the above LQ problem is open-loop solvable
(and hence, by \autoref{thm:open=weak-closed}, weakly closed-loop solvable).
Without loss of generality, we consider only the open-loop solvability at $t=0$.
To this end, let $\e>0$ be arbitrary and consider the Riccati equation \rf{Ric-e}, which, in our example, reads:
\bel{ex-Pe}\left\{\begin{aligned}
   \dot P_\e(s) &={P_\e(s)^2\over\e}, \q s\in[0,1],\\
        P_\e(1) &=1.
\end{aligned}\right.\ee
Solving \rf{ex-Pe} by separating variables, we get
$$ P_\e(s)={\e\over \e+1-s}, \q s\in[0,1]. $$
Let
\begin{align}\label{Th_e}
\Th_\e &\deq-(R+\e I_m+D^\top P_\e D)^{-1}(B^\top P_\e+D^\top P_\e C+S) \nn\\
       &=-{P_\e\over\e} =-{1\over \e+1-s},\qq s\in[0,1].
\end{align}
Then the corresponding BSDE \rf{eta-zeta-e} reads
$$\left\{\begin{aligned}
   d\eta_\e(s) &=-\big[(\Th_\e-1)^\top\eta_\e(s) + \sqrt{2}\z_\e(s) +P_\e b\big]ds +\z_\e(s) dW(s), \q s\in[0,1],\\
    \eta_\e(1) &=0.
\end{aligned}\right.$$
Denote $f(s)={1\over\sqrt{1-s}}$. Using the variation of constants formula for BSDEs, we obtain
\begin{align*}
\eta_\e(s)
  &={\e\over\e+1-s}\,e^{2s-\sqrt{2}W(s)}\,\dbE\lt[\int_s^1 e^{\sqrt{2}W(r)-2r}b(r)dr\bigg|\cF_s\rt] \\
  &={\e\over\e+1-s}\,e^{2s-\sqrt{2}W(s)}\,\dbE\lt[\int_s^1 e^{2\sqrt{2}W(r)-4r}f(r)dr\bigg|\cF_s\rt] \\
  &={\e\over\e+1-s}\,e^{\sqrt{2}W(s)-2s}\int_s^1 f(r)dr, \q s\in[0,1].
\end{align*}
Let
\begin{align}\label{v_e}
v_\e &\deq -(R+\e I_m+D^\top P_\e D)^{-1}(B^\top\eta_\e+D^\top\z_\e+D^\top P_\e\si+\rho) \nn\\
     &= -\e^{-1}\eta_\e =-{1\over\e+1-s}\,e^{\sqrt{2}W(s)-2s}\int_s^1 f(r)dr,\qq s\in[0,1].
\end{align}
Then the corresponding closed-loop system \rf{Xe} can be written as
$$\left\{\begin{aligned}
   dX_\e(s) &= \big\{[\Th_\e(s)-1]X_\e(s)+v_\e(s)+b(s)\big\}ds  \\
            &\hp{=\ } +\sqrt{2}X_\e(s)dW(s), \q s\in[0,1],\\
    X_\e(0) &=x.
\end{aligned}\right.$$
By the variation of constants formula for SDEs, we get
\begin{align*}
X_\e(s) &= (\e+1-s)\,e^{\sqrt{2}W(s)-2s}\int_0^s {1\over\e+1-r}e^{-[\sqrt{2}W(r)-2r]}[v_\e(r)+b(r)]dr \\
&\hp{=\ } +{\e+1-s\over\e+1}\,e^{\sqrt{2}W(s)-2s}x, \q s\in[0,1].
\end{align*}
In light of \autoref{thm:SLQ-oloop-kehua}, to prove the open-loop solvability at $(0,x)$, it suffices to show the family $\{u_\e(\cd)\}_{\e>0}$ defined by
\begin{eqnarray}\label{u_e}
u_\e(s) &\3n\deq\3n& \Th_\e(s)X_\e(s)+v_\e(s) \nn\\
        &\3n=\3n& -e^{\sqrt{2}W(s)-2s}\int_0^s {1\over\e+1-r}e^{-[\sqrt{2}W(r)-2r]}[v_\e(r)+b(r)]dr \nn\\
        &\3n~\3n& -{x\over\e+1}\,e^{\sqrt{2}W(s)-2s} +v_\e(s), \qq s\in[0,1]
\end{eqnarray}
is bounded in $L^2_\dbF(0,1;\dbR)$. For this, let us first simplify \rf{u_e}. By Fubini's theorem,
\begin{align}\label{simplify-1}
   & \int_0^s {1\over \e+1-r}\,e^{-[\sqrt{2}W(r)-2r]}v_\e(r)dr = -\int_0^s {1\over (\e+1-r)^2}\int_r^1 f(\t)d\t dr \nn\\
   &\q= -\int_0^s f(\t)\int_0^\t{1\over (\e+1-r)^2}dr d\t - \int_s^1f(\t)\int_0^s{1\over (\e+1-r)^2}dr d\t\nn\\
   &\q= -\int_0^s {1\over \e+1-\t}f(\t) d\t + {1\over \e+1}\int_0^1 f(\t) d\t -{1\over\e+1-s}\int_s^1 f(\t) d\t.
\end{align}
On the other hand,
\bel{simplify-2} \int_0^s {1\over \e+1-r}\,e^{-[\sqrt{2}W(r)-2r]}b(r)dr =\int_0^s {1\over \e+1-r}f(r) dr. \ee
Substituting \rf{v_e}, \rf{simplify-1} and \rf{simplify-2} into \rf{u_e} yields
\bel{simplify-u_e} u_\e(s) = -{x\over\e+1}\,e^{\sqrt{2}W(s)-2s} -e^{\sqrt{2}W(s)-2s}{1\over\e+1}\int_0^1 f(r)dr
           = -{x+2\over\e+1}e^{\sqrt{2}W(s)-2s}. \ee
A short calculation gives
$$\dbE\int_0^1 |u_\e(s)|^2ds = \lt({x+2\over\e+1}\rt)^2 \les (x+2)^2, \q\forall \e>0.$$
Therefore, $\{u_\e(\cd)\}_{\e>0}$ is bounded in $L^2_\dbF(0,1;\dbR)$. Let $\e\to0$ in \rf{simplify-u_e}, we get an open-loop optimal
control:
$$u^*(s) =-(x+2)e^{\sqrt{2}W(s)-2s},\q s\in[0,1]. $$
\ms

Finally, we let $\e\to0$ in \rf{Th_e} and \rf{v_e} to get a weak closed-loop optimal strategy $(\Th^*(\cd),v^*(\cd))$:
\begin{align*}
  \Th^*(s) &=\lim_{\e\to0}\Th_\e(s) = -{1\over 1-s},  && s\in[0,1),\\
    v^*(s) &=\lim_{\e\to0}  v_\e(s) = -{1\over1-s}\,e^{\sqrt{2}W(s)-2s}\int_s^1 f(r)dr
            =-{2 e^{\sqrt{2}W(s)-2s}\over\sqrt{1-s}}, && s\in[0,1).
\end{align*}
We point out that neither $\Th^*(\cd)$ nor $v^*(\cd)$ is square-integrable on $[0,1)$. Indeed,
\begin{align*}
   \int_0^1 |\Th^*(s)|^2ds   &= \int_0^1 {1\over(1-s)^2}ds=\i, \\
   \dbE\int_0^1 |v^*(s)|^2ds &= \dbE\int_0^1 {4e^{2\sqrt{2}W(s)-4s}\over 1-s}ds = \int_0^1 {4\over 1-s}ds=\i.
\end{align*}
\end{example}

\end{document}